\documentclass[11pt]{amsart}
\usepackage{amsmath,amscd,amsthm,a4wide,amssymb,amsrefs}

\allowdisplaybreaks

\allowdisplaybreaks

\theoremstyle{plain}
\newtheorem{thm}{Theorem}[section]
\newtheorem{lem}[thm]{Lemma}

\newtheorem{prop}[thm]{Proposition}
\theoremstyle{definition}
\newtheorem{rem}[thm]{Remark}

\newtheorem{defi}[thm]{Definition}

\newtheorem{conv}[thm]{Convention}

\numberwithin{equation}{section}

\def\RHS{\operatorname{RHS}}
\def\LHS{\operatorname{LHS}}

\def\qq{\qquad}
\def\rw{\rightarrow}

\def\supp{\operatorname{supp}}

\def\N{\mathbb N}
\def\R{\mathbb R}
\def\Z{\mathbb Z}
\def\z{\mathcal Z}
\def\ap{\approx}

\def\a{\alpha}
\def\b{\beta}

\def\la{\lambda}

\def\vp{\varphi}

\def\i{\infty}
\def\I{(a,b)}

\def\ls{\lesssim}
\def\gs{\gtrsim}

\def\R{\mathbb R}

\begin{document}

\title[Reverse Hardy-type inequalities for supremal operators with measures]
{Reverse Hardy-type inequalities for supremal operators with measures}

\author{R.Ch.~Mustafayev, T.~{\"U}nver}

%





\begin{abstract}
In this paper we characterize the validity of the inequalities
\begin{equation*}\label{eq.0.1.1}
\|g\|_{p,(a,b),\la} \le c \|u(x)
\|g\|_{\i,(x,b),\mu}\|_{q,(a,b),\nu}
\end{equation*}
and
\begin{equation*}\label{eq.0.1.2}
\|g\|_{p,(a,b),\la} \le c \|u(x)
\|g\|_{\i,(a,x),\mu}\|_{q,(a,b),\nu}
\end{equation*}
for all non-negative Borel measurable functions $g$ on the interval
$(a,b) \subseteq \R$, where $0 < p \le +\i$, $0 < q \le +\i$, $\la$,
$\mu$ and $\nu$ are non-negative Borel measures on $(a,b)$, and $u$
is a weight function on $(a,b)$.
\end{abstract}

\maketitle

\section{Introduction}\label{in}

In \cite{EvGogOp}, authors make a comprehensive study of general inequalities of
the form
\begin{equation}\label{eq.1.1.1}
\|gw\|_{p,(a,b),\mu} \le c \left\|u(x) \|g\|_{1,(x,b),\mu}
\right\|_{q,(a,b),\nu}, \qq g \in B^+(I)
\end{equation}
and
\begin{equation}\label{eq.1.1.2}
\|gw\|_{p,(a,b),\mu} \le c \left\|u(x) \|g\|_{1,(a,x),\mu}
\right\|_{q,(a,b),\nu}, \qq g \in B^+(I),
\end{equation}
involving non-negative Borel measures $\mu$, $\nu$ and $\la$, with
complete proofs and estimates for the best constants $c$, provided
that $0 < p \le 1$ and $0 < q \le +\i$. In addition to the extra
generality and the filling  gaps in previous works on these
inequalities, the approach used in \cite{EvGogOp} unifies the
continuous and discrete problems, so that the integral and series
inequalities follow as particular cases. The general inequalities
involving three Borel measures $\la$, $\mu$ and $\nu$
\begin{equation}\label{eq.1.3.1}
\|g\|_{p,(a,b),\la} \le c \left\|u(x) \|g\|_{1,(x,b),\mu}
\right\|_{q,(a,b),\nu}, \qq g \in B^+(I)
\end{equation}
and
\begin{equation}\label{eq.1.3.2}
\|g\|_{p,(a,b),\la} \le c \left\|u(x) \|g\|_{1,(a,x),\mu}
\right\|_{q,(a,b),\nu}, \qq g \in B^+(I),
\end{equation}
are reduced to either to \eqref{eq.1.1.1} or \eqref{eq.1.1.2}.

The object of this paper is to characterize the inequalities
\begin{equation}\label{eq.1.1.3}
\|gw\|_{p,(a,b),\mu} \le c \|u(x)
\|g\|_{\i,(x,b),\mu}\|_{q,(a,b),\nu}
\end{equation}
and
\begin{equation}\label{eq.1.1.4}
\|gw\|_{p,(a,b),\mu} \le c \|u(x)
\|g\|_{\i,(a,x),\mu}\|_{q,(a,b),\nu}
\end{equation}
for all non-negative Borel measurable functions $g$ on the interval
$(a,b) \subseteq \R$, where $0 < p \le +\i$, $0 < q \le +\i$, $\mu$
and $\nu$ are non-negative Borel measures on $(a,b)$. Note that we
do not need the restriction $0 < p \le 1$, which is important when
one consider the reverse Hardy inequalities.  The general
inequalities (involving three non-negative Borel measures $\la$,
$\mu$ and $\nu$) are reduced either to \eqref{eq.1.1.3} or
\eqref{eq.1.1.4}.

The main results of the present paper are Theorems \ref{thm.1.1.1} and \ref{thm.1.2.1}.
Our method is based on a discretization techniques for function norms developed in \cite{EvGogOp}.

The paper is organized as follows. We start with notation and
preliminary results in Section~\ref{s.2}. Necessary and sufficient
conditions for the validity of inequalities \eqref{eq.1.1.3} and
\eqref{eq.1.1.4} can be found in Sections \ref{s.3} and \ref{s.4},
respectively. Finally, in Section \ref{s.5} we show that the results
from Sections \ref{s.3} and \ref{s.4} can be used to characterize
the validity of inequalities mentioned at the Abstract of this
paper.

\section{Notation and preliminaries} \label{s.2}

Throughout the paper, we always denote by  $c$ a positive constant,
which is independent of main parameters but it may vary from line to
line. However a constant with subscript such as $c_1$ does not
change in different occurrences. By $a\lesssim b$ ($b\gtrsim a$), we
mean that $a\leq c b$, where $c > 0$ depends on inessential
parameters. If $a\lesssim b$ and $b\lesssim a$, we write $a\approx
b$ and say that $a$ and $b$ are  {\it equivalent}. We use the
abbreviation $\LHS (*)$ ($\RHS(*)$) for the left (right) hand side
of the relation $(*)$.

We adopt the following usual conventions.
\begin{conv}\label{Notat.and.prelim.conv.1.1}
{\rm (i)} We put $1/(\pm\infty)=0$, $0\cdot(\pm \infty)=0$, $0/0=0$.

{\rm (ii)} We denote by
$$
p':=\begin{cases} \frac p{1-p}&\text{if} \quad 0<p<1,\\
                      +\infty &\text{if}\quad p=1, \\
                   \frac p{p-1}  &\text{if}\quad 1<p<+\infty,\\
          1  &\text{if}\quad p=+\infty.
\end{cases}
$$

{\rm (iii)} If $g$ is a monotone function on $I : = (a,b) \subseteq
\R$, then by $g(a)$ and $g(b)$ we mean the limits $\lim_{x\rw
a+}g(x)$ and $\lim_{x\rw b-}g(x)$, respectively.
\end{conv}

Let $\mu$ be a non-negative Borel measure on $I$. We denote by
$B^+(I)$ the set of all non-negative Borel measurable functions on
$I$. If $E$ is a nonempty Borel measurable subset of $I$ and $f$ is
a Borel measurable function on $E$, then we put
\begin{align*}
\|f\|_{p,E,\mu} & : = \left(\int_{E}|f(y)|^p
d\mu\right)^{{1} / {p}},  \quad \text{if}\quad 0<p<+\infty, \\
\|f\|_{\i,E,\mu} & : =\sup \{\a:~ \mu (\{y \in E:~ |f(y)| \ge \a\})
> 0\}.
\end{align*}

In this paper, $u$, $v$ and $w$ will denote weights, that is,
non-negative Borel measurable functions on $I$.

Let $\emptyset \neq \z \subseteq \overline{\Z} : = \Z \cup \{
-\infty ,+\infty\}$, $0 < q \le +\i$ and $\{w_k\} = \{w_k\}_{k \in
\Z}$ be a sequence of positive numbers. We denote by
$\ell^q(\{w_k\},\z)$ the following discrete analogue of a weighted
Lebesgue space: if $0 < q < + \i$, then
\begin{align*}
&\ell^q(\{w_k\},\z) = \bigg\{ \{a_k\}_{k \in \z}:
\|\{a_k\}\|_{\ell^q(\{w_k\},\z)}
: = \bigg(\sum_{k\in\z}|a_k w_k|^q\bigg)^{1 / q}<+\infty \bigg\},\\
\intertext{and} &
\ell^\infty(\{w_k\},\z) =\left\{ \{a_k\}_{k\in\z}:
\|\{a_k\}\|_{\ell^\infty(\{w_k\},\z)} : =
\sup_{k\in\z}|a_kw_k|<+\infty \right\}.
\end{align*}
If $w_k=1$ for all $k \in \z$, we write simply $\ell^q(\z)$ instead
of $\ell^q(\{w_k\},\z)$. When $N,\,M \in \overline{\Z}$, $N \le M$
and $\z = \{N,N+1,\ldots,M-1,M\}$, we will sometimes use notation
$\ell^q (N,M)$ instead of $\ell^q (\z)$.

We shall use the following inequality, which is a simple consequence
of the discrete H\"{o}lder inequality:
\begin{equation}\label{discrete.Hold.}
\|\{ a_k b_k \}\|_{\ell^q (\z)} \le \|\{ a_k \}\|_{\ell^r (\z)} \|\{
b_k \}\|_{\ell^p (\z)},
\end{equation}
where $1 / r = (1 / q - 1 / p)_+$.\footnote{For any $a\in\R$ denote
by $a_+ = a$ when $a>0$ and $a_+ = 0$ when $a \le 0$.}

\begin{defi} \label{D:2.1}
Let $N,M\in \overline{\Z}$, $N<M$. A positive almost non-increasing
sequence $\{\tau _k\} _{k=N}^M$ (that is, there exists $K \ge 1$
such that $\tau_{n+1} \le K \tau_n$) is called {\it almost
geometrically decreasing} if there are $\alpha \in (1,+\infty )$ and
$L\in \N$ such that
$$
\a \, \tau _k\leq \tau _{k-L} \quad \text{for all} \quad k\in \{
N+L, \dots ,M\}.
$$
A positive almost non-decreasing sequence $\{\sigma _k\} _{k=N}^M$
(that is, there exists $K \ge 1$ such that $\sigma_{n} \le K
\sigma_{n+1}$) is called {\it almost geometrically increasing} if
there are $\alpha \in (1, +\infty )$ and $L\in \N$ such that
$$
\sigma _k\geq \alpha \sigma _{k-L} \quad \text{for all} \quad k\in
\{ N+L, \dots ,M\}.$$
\end{defi}

\begin{rem} \label{R:2.1}
Definition~\ref{D:2.1} implies that  if  $0<q <+\infty$, then the
following three statements are equivalent:

\textup{(i)} \,\,\,\, $\{ \tau_k\} _{k=N}^M \,\,\,\,\,\,
\text{is  an almost geometrically decreasing sequence;} $

\textup{(ii)} \,\, $\{ \tau_k^q\}_{k=N}^M \,\,\,\,\,\, \text{is an
almost geometrically decreasing sequence;}$

 \textup{(iii)} \, $ \{ \tau_k^{-q }\}_{k=N}^M \,\, \text{is an almost
geometrically increasing sequence.}$
\end{rem}

We quote some known results. Proofs  can be found in \cite{Le1} and
\cite{Le2}.
\begin{lem}\label{L:1.2}
Let $q\in (0,+\infty ]$, $N, M \in \overline{\Z}$, $ N\le M$, $\z =
\{N,N+1,\ldots,M-1,M\}$ and let $\{ \tau_k\} _{k=N}^M$ be an almost
geometrically decreasing sequence. Then

\begin{equation}\label{E1.1}
\left\| \left\{\tau _k \sum_{m=N}^{k}a_m\right\}\right\|
_{\ell^q(\z)} \approx \| \{\tau _ka_k\}\| _{\ell^q(\z)}
\end{equation}
and

\begin{equation}\label{E1.2}
\bigg\| \bigg\{\tau _k \sup _{N\leq m\leq k}a_m \bigg\}\bigg\|
_{\ell^q(\z)} \approx \|\{\tau _ka_k\}\| _{\ell^q(\z)}
\end{equation}
for all non-negative sequences $\{ a_k\} _{k=N}^M$.
\end{lem}

Given two (quasi-) Banach spaces $X$ and $Y$, we write $X
\hookrightarrow Y$ if $X \subset Y$ and if the natural embedding of
$X$ in $Y$ is continuous.

The following two lemmas are discrete version of the classical
Landau resonance theorems. Proofs can be found, for example, in
\cite{gp1}.
\begin{prop}\label{prop.2.1}{\rm(\cite[Proposition 4.1]{gp1})}
Let $0 < p,\, q \le +\i$, $\emptyset \neq \z \subseteq
\overline{\Z}$ and let $\{v_k\}_{k\in\z}$ and $\{w_k\}_{k\in\z}$ be
two sequences of positive numbers. Assume that
\begin{equation}\label{eq31-4651}
\ell^p (\{v_k\},\z) \hookrightarrow \ell^q (\{w_k\},\z).
\end{equation}
Then
\begin{equation*}\label{eq31-46519009}
\|\{w_k v_k^{-1}\}\|_{\ell^r(\z)} \le c,
\end{equation*}
where $1 / r = ( 1 / q - 1 / p)_+$ and $c$ stands for the norm of
embedding \eqref{eq31-4651}.
\end{prop}

Now we recall some basic facts on discretization of function norms
from \cite{EvGogOp}.
\begin{lem} \label{L:1.6}{\rm(\cite[Lemma~3.1]{EvGogOp})}
Let $\varphi$  be a non-negative, non-decreasing, finite and
right-continuous function on $(a,b)$. There is a strictly increasing
sequence $\{x_k\}_{k= N}^{M+1}$,  $-\infty\le N\le M\le +\infty$,
with elements from the closure of the interval $(a,b)$,   such that:

\textup{(i)} if $N>-\infty$, then $\varphi(x_N)>0$; $\varphi(x)=0$
for every $x\in (a,x_N)$; if $M<+\infty$, then $x_{M+1}=b$;

\textup{(ii)} $\varphi(x_{k+1}-)\le 2\varphi(x_k)$ \,\,\,\,\, if
\quad $ N\le k\le M$;

\textup{(iii)} $2\varphi(x_{k}-)\leq \varphi(x_{k+1})$ \,\,\, if
\quad $ N< k< M$.
\end{lem}

\begin{defi} \label{L:2.2}{\rm(\cite[Definition~3.2]{EvGogOp})}
Let $\varphi$  be a non-negative, non-decreasing, finite and
right-continuous function on $(a,b)$. A strictly increasing sequence
$\{x_k\}_{k= N}^{M+1}$, $-\infty \le N <M \le +\infty $, is said to
be  {\it a~discretizing sequence of the function} $\varphi$ if it
satisfies the conditions (i) -- (iii) of
 Lemma~\ref{L:1.6}.
\end{defi}
\begin{rem} \label{R:3.2.1}{\rm(\cite[Remark~3.3]{EvGogOp})}
We shall use the following {\it convention}: if $N=-\infty$, then we
put $x_N=\lim_{k\to-\infty}x_k$. It is clear that if   $N=-\infty$
and $x_N>a$, then  $\varphi(x)=0$  for all $x\in (a,x_N)$
(cf.~condition (i) of Lemma~\ref{L:1.6}).
\end{rem}

Let $\vp$ be a non-negative, non-decreasing, finite and
right-continuous function on $(a,b)$. Using a discretizing sequence
$\{x_k\}_{k=N}^{M+1}$ of $\vp$, we define the sequence of intervals
$\{J_k\}_{k=N}^M$ as follows:
\begin{equation}\label{eq.2.2.1}
J_i = (x_i,x_{i+1}],\quad \mbox{if} \quad N \le i < M, \quad
\mbox{and} \quad J_M = (x_M,b) \quad \mbox{if}\quad M < \i.
\end{equation}

\begin{thm}\label{thm.2.2.1}{\rm(\cite[Corollary~3.6 and Corollary~3.7]{EvGogOp})}
Let $0 < q \le +\i$. Suppose that $\mu$ and $\nu$ are non-negative
Borel measures on $I = (a,b)$. Let $u \in B^+(I)$ be such that the
function $\|u\|_{q,(a,t],\nu} < +\i$, $t \in I$. If
$\{x_k\}_{k=N}^{M+1}$ is a discretising sequence of $\vp (t) =
\|u\|_{q,(a,t+],\nu} : = \lim_{s \rw t+} \|u\|_{q,(a,s],\nu}$, $t
\in I$, then
\begin{equation}\label{eq.2.2.2}
\left\|u(x)\|g\|_{\i,(x,b),\mu}\right\|_{q,I,\nu} \ap \left\|
\|g\|_{\i,J_k,\mu} \,\|u\|_{q,(a,x_k+],\nu} \right\|_{\ell^q(N,M)}
\end{equation}
for all $g \in B^+(I)$, where $\{J_k\}_{k=N}^M$ is defined by
\eqref{eq.2.2.1}.
\end{thm}

\begin{rem}\label{lem.2.2.1}
Lemma~\ref{L:1.6} (iii), implies that
$\{\|u\|_{q,(a,x_k+],\nu}\}_{k=N}^{M}$ in Theorem \ref{thm.2.2.1} is
an almost geometrically increasing sequence. (We can take $\alpha
=L=2$ in Definition~\ref{D:2.1}).
\end{rem}

\begin{rem}\label{lem.2.3.1}
Let $q < +\i$. Then
$$
\|u\|_{q,(a,x+],\nu} = \|u\|_{q,(a,x],\nu} \qq\mbox{for all}\qq x
\in I.
$$
\end{rem}

In this paper we shall need the Lebesgue-Stieltjes integral. To this
end, we recall some basic facts.

Let $\vp$ be non-decreasing and finite function on the interval $I :
= (a,b)\subseteq \R$. We assign to $\vp$ the function $\la$ defined
on subintervals of $I$ by
\begin{align}
\la ([\a,\b]) & = \vp (\b+) - \vp(\a-), \label{eq+-}\\
\la ([\a,\b)) & = \vp (\b-) - \vp(\a-), \label{eq--}\\
\la ((\a,\b]) & = \vp (\b+) - \vp(\a+), \label{eq++}\\
\la ((\a,\b)) & = \vp (\b-) - \vp(\a+). \label{eq-+}
\end{align}
The function $\la$ is a non-negative, additive and regular function
of intervals. Thus (cf. \cite{ru}, Chapter 10), it admits a unique
extension to a non-negative Borel measure $\la$ on $I$. The
Lebesgue-Stieltjes integral $\int_I f\,d\vp$ is defined as $\int_I
f\,d\la$.

If $J\subseteq I$, then the Lebesgue-Stieltjes integral $\int_J
f\,d\vp$ is defined as $\int_J f\,d\la$. We shall also use the
Lebesgue-Stieltjes integral $\int_J f\,d\vp$ when $\vp$ is a
non-increasing and finite on the interval $I$. In such a case we put
$$
\int_J f\,d\vp : = - \int_J f\,d(-\vp).
$$

If $\vp$ is a non-decreasing, finite and right-continuous function
on $I = (a,b)$ and $J$ is a subinterval of $I$ of the form
$(\a,\b)$, $[\a,\b)$ or $(\a,\b]$, then the formulae \eqref{eq-+},
\eqref{eq--} and \eqref{eq++} imply that
\begin{align}
\int_{(\a,\b)} d\vp & = \vp (\b-) - \vp(\a), \label{eqq+-}\\
\int_{[\a,\b)} d\vp & = \vp (\b-) - \vp(\a-), \label{eqq--}\\
\int_{(\a,\b]} d\vp & = \vp (\b) - \vp(\a). \label{eqq++}
\end{align}

In this paper the role of the function $\vp$ will be played by a
function $h$ which will be {\it non-decreasing} and {\it
right-continuous} or {\it non-increasing} and {\it left-continuous}
on $I$. At the first case, the associated Borel measure $\la$ will
be determined by (cf. \eqref{eq++})
\begin{equation}\label{eq.alphabeta+}
\la((\a,\b]) = h(\b) - h(\a) \qq \mbox{for any}\qq (\a,\b] \subset I
\end{equation}
(since the Borel subsets of $I$ can be generated by subintervals
$(\a,\b] \subset I$).

Considering inequalities \eqref{eq.1.1.3} and \eqref{eq.1.1.4}, in
the case when $0 < p < q \le +\i$ and $1 / r = 1 / p - 1 / q$, we
shall write conditions characterizing the validity of inequalities
in a compact form involving $\int_{(a,b)}f\,dh$. To this end, we
adopt the following conventions from \cite{EvGogOp}.
\begin{conv}\label{conv.1}
Let $I = (a,b) \subseteq \R$, $f: I \rw [0,\i]$ and $h: I \rw
[-\i,0]$. Assume that $h$ is non-decreasing and right-continuous on
$I$. If $h: I \rw (-\i,0]$, then the symbol $\int_I f\,dh$ means the
usual Lebesgue-Stieltjes integral. However, if $h = -\i$ on some
subinterval $(a,c)$ with $c \in I$, then we define $\int_I f\,dh$
only if $f = 0$ on $(a,c]$ and we put
$$
\int_I f\,dh = \int_{(c,b)} f\,dh.
$$
\end{conv}

\begin{conv}\label{conv.2}
Let $I = (a,b) \subseteq \R$, $f: I \rw [0,\i]$ and $h: I \rw
[0,\i]$. Assume that $h$ is non-decreasing and left-continuous on
$I$. If $h: I \rw [0,\i)$, then the symbol $\int_I f\,dh$ means the
usual Lebesgue-Stieltjes integral. However, if $h = +\i$ on some
subinterval $(c,b)$ with $c \in I$, then we define $\int_I f\,dh$
only if $f = 0$ on $[c,b)$ and we put
$$
\int_I f\,dh = \int_{(a,c)} f\,dh.
$$
\end{conv}

\section{Reverse Hardy-type inequalities for supremal
operators}\label{s.3}

In this section we characterize inequality \eqref{eq.1.1.3}. We start with the following discretization lemma.
\begin{lem}\label{lem.3.1.1}
Assume that $0 < p,\,q \le +\i$. Let $\mu$ and $\nu$ be non-negative
Borel measures on $I = (a,b) \subseteq \R$. Let $w \in B^+ (I)$ and
let $u \in B^+(I)$ satisfy $\|u\|_{q,(a,t],\nu} < \i$ for all $t \in
I$ and $u \neq 0$ a.e. on $(a,b)$. If $\{x_k\}_{k=N}^{M+1}$ is a
discretising sequence of $\vp (t) : = \|u\|_{q,(a,t+],\nu}$, then
inequality \eqref{eq.1.1.3} holds for all $g \in B^+ (I)$ if and
only if
\begin{equation}\label{eq.3.1.1}
A : = \left\| \left\{ \|w\|_{p,J_k,\mu}\,\|u\|_{q,(a,x_k+],\nu}^{-1}
\right\} \right\|_{\ell^{\rho}(N,M)} < \i,
\end{equation}
and
\begin{equation}\label{eq.3.1.2}
w = 0 \quad \mu - \mbox{a.e. in} ~ (a,x_N] \quad \mbox{ if} \quad
x_N > a,
\end{equation}
where ${1} / {\rho} : = \left( {1} / {p} - 1 / q \right)_+$.

The best possible constant $c$ in \eqref{eq.1.1.3} satisfies $c \ap
A$.
\end{lem}

\begin{proof}
By Theorem \ref{thm.2.2.1},
\begin{equation}\label{eq.5.1.2}
\left\|u(x)\|g\|_{\i,(x,b),\mu}\right\|_{q,I,\nu} \ap \left\|
\left\{\|g\|_{\i,J_k,\mu}
\|u\|_{q,(a,x_k+],\nu}\right\}\right\|_{\ell^q(N,M)},
\end{equation}
for all $g \in B^+(I)$, where $\{x_k\}_{k=N}^{M+1}$ is a
discretising sequence of the function $\vp (t) =
\|u\|_{q,(a,t+],\nu}$, $t \in (a,b)$, and $\{J_k\}_{k=N}^M$ is
defined by \eqref{eq.2.2.1}. By Lemma \ref{L:1.6} (cf. also Remark
\ref{R:3.2.1}),
\begin{align}\label{eq.5.1.3}
\text{if}\quad & x_N>a, \quad \text{then} \quad \|u\|_{q,(a,x_N),\nu}=0;\\
\text{if}\quad & M<+\infty, \quad \text{then}\quad  x_{M+1}=b;
\notag
\end{align}
\begin{align}
&\|u\|_{q,(a,x_{k+1}),\nu} \leq 2\|u\|_{q,(a,x_k+],\nu}
\quad \text{if}\quad N \leq k \leq M \label{eq.5.1.4};\\
2 & \|u\|_{q,(a,x_k),\nu} \leq \|u\|_{q,(a,x_{k+1}+],\nu} \quad \,\,
\, \text{if} \quad N < k < M. \label{eq.5.1.5}
\end{align}
{\bf Sufficiency.} Let \eqref{eq.3.1.1} and \eqref{eq.3.1.2} hold.
Since
\begin{equation}\label{eq.5.1.7}
\|gw\|_{p,(a,b),\mu} = \left\| \left\{\|gw\|_{p,J_k,\mu} \right\}
\right\|_{\ell^p(N,M)}, \quad \mbox{for any}\quad g \in B^+(I),
\end{equation}
and
\begin{equation}\label{eq.5.1.8}
\|gw\|_{p,J_k,\mu} \le \|g\|_{\i,J_k,\mu} \|w\|_{p,J_k,\mu}, \quad N
\le k \le M,
\end{equation}
on using \eqref{discrete.Hold.} and \eqref{eq.5.1.2}, we have that
\begin{align*}
\|gw\|_{p,(a,b),\mu} \le & \left\| \left\{\|g\|_{\i,J_k,\mu}
\|w\|_{p,J_k,\mu} \right\}
\right\|_{\ell^p(N,M)} \\
\le & \left\| \left\{\|w\|_{p,J_k,\mu}
\|u\|_{1,(a,x_k],\nu}^{-1}\right\}\right\|_{\ell^{\rho}(N,M)}
\,\left\| \left\{\|g\|_{\i,J_k,\mu}
\|u\|_{q,(a,x_k+],\nu}\right\}\right\|_{\ell^q(N,M)} \\
\ap &  \left\| \left\{\|w\|_{p,J_k,\mu}
\|u\|_{q,(a,x_k+],\nu}^{-1}\right\}\right\|_{\ell^{\rho}(N,M)}\,
\left\|u(x)\|g\|_{\i,(x,b),\mu}\right\|_{q,I,\nu}
\end{align*}
Consequently, $c \ls A$.

{\bf Necessity.} We now prove necessity. The validity of inequality
\eqref{eq.1.1.3} on $B^+(I)$ and \eqref{eq.5.1.2} imply that
\begin{equation}\label{eq.5.1.11}
\left\| \left\{\|gw\|_{p,J_k,\mu} \right\} \right\|_{\ell^p(N,M)}
\ls c \,\left\| \left\{\|g\|_{\i,J_k,\mu}
\|u\|_{q,(a,x_k+],\nu}\right\}\right\|_{\ell^q(N,M)}
\end{equation}
for all $g \in B^+(I)$.

Let $g_k \in B^+(I)$, $N \le k \le M$, be functions such that
\begin{equation}\label{eq.5.1.12}
\supp g_k \subset J_k, \quad \|g_k\|_{\i,J_k,\mu} = 1 \quad
\mbox{and} \quad \|gw\|_{p,J_k,\mu} \gs \|w\|_{p,J_k,\mu}.
\end{equation}
Then we define the test function $g$ by
\begin{equation}\label{eq.5.1.13}
g = \sum_{k=N}^M a_k g_k,
\end{equation}
where $\{a_k\}$ is a sequence of non-negative numbers. Consequently,
\eqref{eq.5.1.11} yields
\begin{equation}\label{eq.5.1.14}
\left\| \left\{a_k\|w\|_{p,J_k,\mu} \right\} \right\|_{\ell^p(N,M)}
\ls c \,\left\| \left\{a_k
\|u\|_{q,(a,x_k+],\nu}\right\}\right\|_{\ell^q(N,M)},
\end{equation}
and, by Proposition \ref{prop.2.1}, we arrive at
\begin{equation}\label{eq.5.1.15}
A = \left\| \left\{ \|w\|_{p,J_k,\mu}\,\|u\|_{q,(a,x_k+],\nu}^{-1}
\right\} \right\|_{\ell^{\rho}(N,M)} \ls c.
\end{equation}

On the other hand, assuming that $x_N > a$, testing \eqref{eq.1.1.3}
with $g = \chi_{(a,x_N]}$ and using \eqref{eq.5.1.3}, we arrive at
$\|w\|_{p,(a,x_N],\mu} = 0$, which implies \eqref{eq.3.1.2}.
\end{proof}

The following lemma is true.
\begin{lem}\label{lem.3.2.2}
Assume that $0 < q \le p \le +\i$. Let $\mu$ and $\nu$ be
non-negative Borel measures on $I = (a,b) \subseteq \R$. Let $w \in
B^+ (I)$ and let $u \in B^+(I)$ satisfy $\|u\|_{q,(a,t],\nu} < \i$
for all $t \in I$ and $u \neq 0$ a.e. on $(a,b)$. If
$\{x_k\}_{k=N}^{M+1}$ is a discretising sequence of $\vp (t) =
\|u\|_{q,(a,t+],\nu}$, $t \in I$, then
\begin{equation}\label{eq.3.9.1}
A = \left\| \left\{ \|w\|_{p,J_k,\mu}\,\|u\|_{q,(a,x_k+],\nu}^{-1}
\right\} \right\|_{\ell^{\i} (N,M)} < \i,
\end{equation}
and \eqref{eq.3.1.2} hold if and only if
\begin{equation}\label{eq.3.9.2}
A_1 : = \left\| \|w\|_{p,(a,x],\mu}
\|u\|_{q,(a,x),\nu}^{-1}\right\|_{\i,\I,\mu} < \i.
\end{equation}
Moreover, $A \ap A_1$.
\end{lem}
\begin{proof}
{\bf Sufficiency.} Assume that $A_1 <\i$. This condition and
\eqref{eq.5.1.3} imply that
\begin{equation}\label{eq.5.1.6}
\|w\|_{p,(a,x_N],\mu} = 0 \quad \mbox{if}\quad x_N > a.
\end{equation}
Consequently, \eqref{eq.3.1.2} holds.

Applying \eqref{eq.5.1.4}, we get that
\begin{align*}
A = \sup_{N \le k \le M} \|w\|_{p,J_k,\mu}
\|u\|_{q,(a,x_k+],\nu}^{-1}
\le &  \, 2 \sup_{N \le k \le M} \|w\|_{p,J_k,\mu} \|u\|_{q,(a,x_{k+1}),\nu}^{-1} \label{eq.5.1.10} \\
\le & \, 2 \sup_{N \le k \le M} \|w\|_{p,(a,x_{k+1}] \cap I,\mu}
\|u\|_{q,(a,x_{k+1}),\nu}^{-1} \notag \\
\le & 2 \, A_1. \notag
\end{align*}

{\bf Necessity.} Assume that \eqref{eq.3.9.1} and \eqref{eq.3.1.2}
hold. Therefore, on using \eqref{eq.2.2.1},
$$
A_1 = \sup_{N \le k\le M}  \left\| \|w\|_{p,(a,x],\mu}
\|u\|_{q,(a,x),\nu}^{-1}\right\|_{\i,J_k,\mu}
$$
and hence
\begin{align*}
A_1 \le & \sup_{N \le k\le M} \|w\|_{p,(a,x_{k+1}]\cap I,\mu}
\left\| \|u\|_{q,(a,x),\nu}^{-1}\right\|_{\i,J_k,\mu} \\
\le & \sup_{N \le k\le M} \|w\|_{p,(a,x_{k+1}]\cap I,\mu}
\|u\|_{q,(a,x_k+],\nu}^{-1}.
\end{align*}
Applying \eqref{eq.3.1.2} again, on using the fact that
$\{\|u\|_{q,(a,x_k+],\nu}^{-1}\}_{k=N}^M$ is almost geometrically
decreasing and Lemma \ref{L:1.2}, we obtain that
$$
A_1 \ls \sup_{N \le k \le M} \|w\|_{p,J_k,\mu}
\|u\|_{q,(a,x_k+],\nu}^{-1} = A.
$$
\end{proof}

To prove our main statement we need the following lemma.
\begin{lem}\label{lem.3.2.1}
Assume that $0 < p < q \le +\i$ and $1 / r = 1 / p - 1 / q$. Let
$\mu$ and $\nu$ be non-negative Borel measures on $I = (a,b)
\subseteq \R$. Let $w \in B^+ (I)$ and let $u \in B^+(I)$ satisfy
$\|u\|_{q,(a,t],\nu} < \i$ for all $t \in I$ and $u \neq 0$ a.e. on
$(a,b)$. If $\{x_k\}_{k=N}^{M+1}$ is a discretising sequence of $\vp
(t) = \|u\|_{q,(a,t+],\nu}$, $t \in I$, then
\begin{equation}\label{eq.3.2.1}
A = \left\| \left\{ \|w\|_{p,J_k,\mu}\,\|u\|_{q,(a,x_k+],\nu}^{-1}
\right\} \right\|_{\ell^r (N,M)} < \i,
\end{equation}
and \eqref{eq.3.1.2} hold if and only if
\begin{equation}\label{eq.3.2.2}
A_2 : = \left( \int_{(a,b)} \|w\|_{p,(a,x],\mu}^{r} d \left( -
\|u\|_{q,(a,x+],\nu}^{-r} \right) \right)^{1 / r} +
\|w\|_{p,(a,b),\mu} \|u\|_{q,(a,b),\nu}^{-1} < \i.
\end{equation}
Moreover, $A \ap A_2$.
\end{lem}

\begin{proof}
Let $\{x_k\}_{k=N}^{M+1}$ be a discretising sequence of the function
$\vp (t) = \|u\|_{q,(a,t+],\nu}$, $t \in (a,b)$, and
$\{J_k\}_{k=N}^M$ is defined by \eqref{eq.2.2.1}. By Lemma
\ref{L:1.6} (cf. also Remark \ref{R:3.2.1}), \eqref{eq.5.1.3}-\eqref{eq.5.1.5} hold.

{\bf Sufficiency.} Assume that $A_2 < \i$. This condition,
\eqref{eq.5.1.3} and Convention \ref{conv.1} imply that
\eqref{eq.3.1.2} holds. By \eqref{eq.5.1.5},
$$
2 \|u\|_{q,(a,x_{k+1}),\nu} \le \|u\|_{q,(a,x_{k+2}+],\nu} \le
\|u\|_{q,(a,x_{k+3}),\nu} \qq \mbox{if}\qq N < k+1 < M.
$$
Therefore,
$$
\|u\|_{q,(a,x_{k+3}),\nu}^{-r} \le 2^{-r}\,
\|u\|_{q,(a,x_{k+1}),\nu}^{-r},
$$
which yields
$$
\|u\|_{q,(a,x_{k+1}),\nu}^{-r} - \|u\|_{q,(a,x_{k+3}),\nu}^{-r} \ge
(1 - 2^{-r}) \|u\|_{q,(a,x_{k+1}),\nu}^{-r} \qq \mbox{if}\qq N \le k
\le M-2.
$$
Assume that $N \le M-2$. On using \eqref{eq.5.1.4} and the last
estimate, we arrive at
\begin{align}
A^r \ls & \sum_{k = N}^M \|w\|_{p,J_k,\mu}^{r}
\|u\|_{q,(a,x_{k+1}),\nu}^{-r} \notag \\
\ls & \sum_{k = N}^{M-2} \|w\|_{p,J_k,\mu}^{r} \left(
\|u\|_{q,(a,x_{k+1}),\nu}^{-r} - \|u\|_{q,(a,x_{k+3}),\nu}^{-r}
\right) \notag \\
&  +  \|w\|_{p,J_{M-1},\mu}^{r} \left(
\|u\|_{q,(a,x_M),\nu}^{-r} -
\|u\|_{q,(a,b),\nu}^{-r} \right) \notag \\
&  + \|w\|_{p,J_{M-1},\mu}^{r}
\|u\|_{q,(a,b),\nu}^{-r} + \|w\|_{p,J_M,\mu}^{r}
\|u\|_{q,(a,b),\nu}^{-r}. \label{eq.8.1.4}
\end{align}
Now, by \eqref{eqq--} with $\vp (t) = - \|u\|_{q,(a,t+],\nu}^{-r}$,
$t \in I$, and $[\a,\b) = [x_{k+1},x_{k+3})$, $N \le k \le M-2$, or
$[\a,\b) = [x_M,b)$, we obtain that
\begin{align*}
A^r \ls & \sum_{k = N}^{M-2} \|w\|_{p,J_k,\mu}^{r}
\int_{[x_{k+1},x_{k+3})} d\left( - \|u\|_{q,(a,t+],\nu}^{-r} \right) \\
&  +  \|w\|_{p,J_{M-1},\mu}^{r} \int_{[x_M,b)}
d\left( - \|u\|_{q,(a,t+],\nu}^{-r} \right) + 2
\|w\|_{p,(a,b),\mu}^{r} \|u\|_{q,(a,b),\nu}^{-r} \\
\ls & \sum_{k = N}^{M-2}
\int_{[x_{k+1},x_{k+3})} \|w\|_{p,(a,t],\mu}^{r}\,d\left( - \|u\|_{q,(a,t+],\nu}^{-r} \right) \\
& +  \int_{[x_M,b)} \|w\|_{p,(a,t],\mu}^{r}
\,d\left( - \|u\|_{q,(a,t+],\nu}^{-r} \right) + 2
\|w\|_{p,(a,b),\mu}^{r} \|u\|_{q,(a,b),\nu}^{-r} \\
\le & \int_{(a,b)} \|w\|_{p,(a,t],\mu}^{r} \,d\left( -
\|u\|_{q,(a,t+],\nu}^{-r} \right) + 2
\|w\|_{p,(a,b),\mu}^{r} \|u\|_{q,(a,b),\nu}^{-r} \\
\ls & ~A_2^{r}
\end{align*}
(note that we have used \eqref{eq.3.1.2} and Convention
\ref{conv.1}), that is,
\begin{equation}\label{eq.8.1.5}
A \ls A_2.
\end{equation}
If $N > M-2$, then \eqref{eq.8.1.5} can be proved more simply and we
omit the proof.

{\bf Necessity.} Now assume that $A < \i$ and \eqref{eq.3.1.2}
holds. On using \eqref{eq.3.1.2}, together with \eqref{eqq++} and
\eqref{eqq+-}, we have that
\begin{align}
A_2^r \ap & \sum_{k=N}^M \int_{J_k} \|w\|_{p,(a,x],\mu}^r d \left( -
\|u\|_{q,(a,x+],\nu}^{-r} \right) + \|w\|_{p,(a,b),\mu}^r
\|u\|_{q,(a,b),\nu}^{-r} \notag
\\
\le & \sum_{k=N}^{M-1} \|w\|_{p,(a,x_{k+1}],\mu}^r \int_{J_k} d
\left( - \|u\|_{q,(a,x+],\nu}^{-r} \right) \notag \\
& + \|w\|_{p,(a,b),\mu}^r \int_{(x_M,b)} d \left( -
\|u\|_{q,(a,x+],\nu}^{-r} \right) + \|w\|_{p,(a,b),\mu}^r
\|u\|_{q,(a,b),\nu}^{-r} \notag \\
= & \sum_{k=N}^{M-1} \|w\|_{p,(a,x_{k+1}],\mu}^r \left(
\|u\|_{q,(a,x_k+],\nu}^{-r} - \|u\|_{q,(a,x_{k+1}+],\nu}^{-r}\right)
\notag \\
& + \|w\|_{p,(a,b),\mu}^r \left( \|u\|_{q,(a,x_M+],\nu}^{-r} -
\|u\|_{q,(a,b),\nu}^{-r}\right) + \|w\|_{p,(a,b),\mu}^r
\|u\|_{q,(a,b),\nu}^{-r} \notag \\
\ls & \sum_{k=N}^{M-1} \|w\|_{p,(a,x_{k+1}],\mu}^r
\|u\|_{q,(a,x_k+],\nu}^{-r}  + \|w\|_{p,(a,b),\mu}^r
\|u\|_{q,(a,x_M+],\nu}^{-r}. \label{eq.8.1.6}
\end{align}
Thus, using \eqref{eq.3.1.2} again, we arrive at
\begin{align*}
A_2^r \ls & \sum_{k=N}^M \|w\|_{p,(a,x_{k+1}]\cap I,\mu}^r
\|u\|_{q,(a,x_k+],\nu}^{-r} \\
= & \sum_{k=N}^M \left( \sum_{i = N}^ k \|w\|_{p,J_i,\mu}^r
\right)\|u\|_{q,(a,x_k+],\nu}^{-r} \\
\end{align*}
Now, the fact that $\{\|u\|_{q,(a,x_k+],\nu}^{-r}\}_{k=N}^M$ is
almost geometrically decreasing and Lemma \ref{L:1.2} imply that
\begin{equation}\label{eq.8.1.7}
A_2 \ls  \left( \sum_{k=N}^M \|w\|_{p,J_k,\mu}^r
\|u\|_{q,(a,x_k+],\nu}^{-r} \right)^{1/r} = A.
\end{equation}
Combining \eqref{eq.8.1.5} and \eqref{eq.8.1.7}, we get $A \ap A_2$.
\end{proof}

Now we are in position to prove our first main result.
\begin{thm}\label{thm.1.1.1}
Assume that $0 < p,\,q \le +\i$. Let $\mu$ and $\nu$ be non-negative
Borel measures on $I = (a,b) \subseteq \R$. Let $w \in B^+ (I)$ and
let $u \in B^+(I)$ satisfy $\|u\|_{q,(a,t],\nu} < \i$ for all $t \in
I$ and $u \neq 0$ a.e. on $(a,b)$.

{\rm (i)} Let $0 < q \le p \le +\i$. Then inequality
\eqref{eq.1.1.3} holds for all $g \in B^+ (I)$ if and only if
\begin{equation}\label{eq.1.2.1}
A_1 = \left\| \|w\|_{p,(a,x],\mu}
\|u\|_{q,(a,x),\nu}^{-1}\right\|_{\i,\I,\mu} < \i.
\end{equation}
The best possible constant $c$ in \eqref{eq.1.1.3} satisfies $c \ap
A_1$.

{\rm (ii)} Let $0 < p < q < +\i$ and $1 / r = 1 / p - 1 / q$. Then
inequality \eqref{eq.1.1.3} holds for all $g \in B^+ (I)$ if and
only if
\begin{equation}\label{eq.1.2.2}
A_2 = \left( \int_{(a,b)} \|w\|_{p,(a,x],\mu}^{r} d \left( -
\|u\|_{q,(a,x],\nu}^{-r} \right) \right)^{1 / r} +
\|w\|_{p,(a,b),\mu} \|u\|_{q,(a,b),\nu}^{-1} < \i.
\end{equation}
The best possible constant $c$ in \eqref{eq.1.1.3} satisfies $c \ap
A_2$.

{\rm (iii)} Let $0 < p < +\i$, $q = +\i$. Then inequality
\eqref{eq.1.1.3} holds for all $g \in B^+ (I)$ if and only if
\begin{align}
A_3 & = \left( \int_{(a,b)} \left(
\frac{w(x)}{\|u\|_{\i,(a,x),\nu}}\right)^p\,d\mu(x) \right)^{1 / p}
\label{eq.1.2.3} \\
& \ap \left( \int_{(a,b)} \|w\|_{p,(a,x],\mu}^{p} d \left( -
\|u\|_{\i,(a,x+],\nu}^{-p} \right) \right)^{1 / p} +
\|w\|_{p,(a,b),\mu} \|u\|_{\i,(a,b),\nu}^{-1} < \i. \notag
\end{align}
The best possible constant $c$ in \eqref{eq.1.1.3} satisfies $c \ap
A_3$.
\end{thm}

\begin{proof}
(i) Let $0 < q \le p \le +\i$. The statement follows by Lemmas
\ref{lem.3.1.1} and \ref{lem.3.2.2}.

(ii) Let $0 < p < q < +\i$. The
statement follows by Lemmas \ref{lem.3.1.1} and \ref{lem.3.2.1}.

{\rm (iii)} Let $0 < p < q = +\i$. The statement follows by Lemmas
\ref{lem.3.1.1}, \ref{lem.3.2.1} and an integration by parts
formula.
\end{proof}

\begin{rem}
Let $q < +\i$. Since
$$
\|u\|_{q,(a,x+],\nu} = \|u\|_{q,(a,x],\nu} \qq\mbox{for all}\qq x
\in I,
$$
the cases (ii) and (iii) can be combined:

{$\rm (ii)^{\prime}$} Let $0 < p < q \le +\i$ and $1 / r = 1 / p - 1
/ q$. Then inequality \eqref{eq.1.1.3} holds for all $g \in B^+ (I)$
if and only if
\begin{equation*}
A_2^{\prime} : = \left( \int_{(a,b)} \|w\|_{p,(a,x],\mu}^{r} d
\left( - \|u\|_{q,(a,x+],\nu}^{-r} \right) \right)^{1 / r} +
\|w\|_{p,(a,b),\mu} \|u\|_{q,(a,b),\nu}^{-1} < \i.
\end{equation*}
The best possible constant $c$ in \eqref{eq.1.1.3} satisfies $c \ap
A_2^{\prime}$.
\end{rem}

\begin{rem}
Note that inequality
\eqref{eq.1.1.3} can be easily characterized by a more simply argumentations when $q = +\i$. Exchanging essential suprema, we have that
\begin{equation}\label{eq.8.1.1}
\|u(x)\|g\|_{\i,(x,b),\mu}\|_{\i,(a,b),\nu} =
\|\|u\|_{\i,(a,x),\nu}g(x)\|_{\i,(a,b),\mu},\qquad g \in B^+(I).
\end{equation}
Consequently, \eqref{eq.1.1.3} is nothing else the description of
the embeddings of weighted $L^{\i}(\mu)$ to weighted $L^p(\nu)$
(see, for instance, \cite[Proposition 6.13]{F}). Indeed: on using \eqref{eq.8.1.1}, for the
best constant $c$ in \eqref{eq.1.1.3} we have that
\begin{align*}
c = & \sup_{g \not\sim 0}
\frac{\|gw\|_{p,(a,b),\mu}}{\|u(x)\|g\|_{\i,(x,b),\mu}\|_{\i,(a,b),\nu}}
\\
= & \sup_{g \not\sim 0}
\frac{\|gw\|_{p,(a,b),\mu}}{\|\|u\|_{\i,(a,x),\nu}g(x)\|_{\i,(a,b),\mu}}
\\
= & \left( \int_{(a,b)} \left(
\frac{w(x)}{\|u\|_{\i,(a,x),\nu}}\right)^p\,d\mu(x) \right)^{1 / p},
\end{align*}
when $0 < p < +\i$, and
\begin{align*}
c = & \left\| w(x) \|u\|_{\i,(a,x),\nu}^{-1} \right\|_{\i,\I,\mu}
\\
= &  \left\| \|w\|_{\i,(a,x],\mu}
\|u\|_{\i,(a,x),\nu}^{-1}\right\|_{\i,\I,\mu},
\end{align*}
when $p = +\i$. In the last equality, exchanging essential suprema,
we have used that
\begin{align*}
\left\| w(x) \|u\|_{\i,(a,x),\nu}^{-1} \right\|_{\i,\I,\mu} & =
\left\| w(x) \left\|\|u\|_{\i,(a,t),\nu}^{-1}\right\|_{\i,[x,b),\mu}
\right\|_{\i,\I,\mu} \\
& = \left\|
\left\|w(x)\|u\|_{\i,(a,t),\nu}^{-1}\right\|_{\i,[x,b),\mu}
\right\|_{\i,\I,\mu} \\
& = \left\|
\left\|w(x)\chi_{[x,b)}(t)\|u\|_{\i,(a,t),\nu}^{-1}\right\|_{\i,\I,\mu}
\right\|_{\i,\I,\mu} \\
& = \left\| \left\|w(x)\chi_{[x,b)}(t)\right\|_{\i,\I,\mu}
\|u\|_{\i,(a,t),\nu}^{-1}
\right\|_{\i,\I,\mu} \\
& = \left\| \|w\|_{\i,(a,t],\mu} \|u\|_{\i,(a,t),\nu}^{-1}
\right\|_{\i,\I,\mu}.
\end{align*}
(see, \cite[Remark 4.2]{EvGogOp}).
\end{rem}

\section{Reverse inequality for the dual operator}\label{s.4}

The aim of this section is to characterize the inequality
\eqref{eq.1.1.4}.
\begin{thm}\label{thm.1.2.1}
Assume that $0 < p,\,q \le +\i$. Let $\mu$ and $\nu$ be non-negative
Borel measures on $I = (a,b) \subseteq \R$. Let $w \in B^+ (I)$ and
let $u \in B^+(I)$ satisfy $\|u\|_{q,[t,b),\nu} < \i$ for all $t \in
I$ and $u \neq 0$ a.e. on $(a,b)$.

{\rm (i)} Let $0 < q \le p \le +\i$. Then inequality
\eqref{eq.1.1.4} holds for all $g \in B^+ (I)$ if and only if
\begin{equation}\label{eq.1.1}
B_1 : = \left\| \|w\|_{p,[x,b),\mu}
\|u\|_{q,(x,b),\nu}^{-1}\right\|_{\i,\I,\mu} < \i.
\end{equation}
The best possible constant $c$ in \eqref{eq.1.1.4} satisfies $c \ap
B_1$.

{\rm (ii)} Let $0 < p < q < +\i$ and $1 / r = 1 / p - 1 / q$. Then
inequality \eqref{eq.1.1.4} holds for all $g \in B^+ (I)$ if and
only if
\begin{equation}\label{eq.1.2}
B_2 : = \left( \int_{(a,b)} \|w\|_{p,[x,b),\mu}^{r} d \left(
\|u\|_{q,[x,b),\nu}^{-r} \right) \right)^{1 / r} +
\|w\|_{p,(a,b),\mu} \|u\|_{q,(a,b),\nu}^{-1} < \i.
\end{equation}
The best possible constant $c$ in \eqref{eq.1.1.4} satisfies $c \ap
B_2$.

{\rm (iii)} Let $0 < p < +\i$, $q = +\i$. Then inequality
\eqref{eq.1.1.4} holds for all $g \in B^+ (I)$ if and only if
\begin{align}
B_3 : & = \left( \int_{(a,b)} \left(
\frac{w(x)}{\|u\|_{\i,(x,b),\nu}}\right)^p\,d\mu(x) \right)^{1 / p}
\label{eq.1.3} \\
& \ap \left( \int_{(a,b)} \|w\|_{p,[x,b),\mu}^{p} d \left(
\|u\|_{\i,[x-,b),\nu}^{-p} \right) \right)^{1 / p} +
\|w\|_{p,(a,b),\mu} \|u\|_{\i,(a,b),\nu}^{-1} < \i, \notag
\end{align}
where
$$
\|u\|_{\i,[x-,b),\nu} : = \lim_{s\rw t-} \|u\|_{\i,[s,b),\nu}.
$$
The best possible constant $c$ in \eqref{eq.1.1.4} satisfies $c \ap
B_3$.
\end{thm}
\begin{rem}
Note that the proof of Theorem \ref{thm.1.2.1} is similar to the
proof of Theorems 5.1 and 5.4 from \cite{EvGogOp}. For the sake of
completeness we give complete proof here.
\end{rem}

\begin{proof}
If $\la$ is a non-negative Borel measure on $I$, we denote by
$\tilde{\la}$ a non-negative Borel measure on $\tilde{I} : =
(-b,-a)$ defined by
$$
\tilde{\la} : = \la (-E), \qq\mbox{where}\qq -E : = \{-x:~x\in E\}.
$$
Similarly, if $h \in B^+(I)$, then the function $\tilde{h} \in
B^+(\tilde{I})$ is given by
$$
\tilde{h} : = h (-x),\qq x \in \tilde{I}.
$$
It is clear that
\begin{equation}\label{eq.10.2.1}
\int_E h \,d\la = \int_{-E} \tilde{h}\,d\tilde{\la},
\end{equation}
and
\begin{equation}\label{eq.10.2.2}
\|h\|_{\i,E,\la} = \|\tilde{h}\|_{\i,-E,\tilde{\la}},
\end{equation}
for any Borel subset $E$ of $I$. In particular,
\begin{align*}
\|gw\|_{p,(a,b),\mu} & =
\|\tilde{g}\tilde{w}\|_{p,(-b,-a),\tilde{\mu}},  \\
\|u(x)\|g\|_{\i,(a,x),\mu}\|_{q,(a,b),\nu} & =
\|\tilde{u}(x)\|g\|_{\i,(a,-x),\mu}\|_{q,(-b,-a),\tilde{\nu}}, \\
& =
\|\tilde{u}(x)\|\tilde{g}\|_{\i,(x,-a),\tilde{\mu}}\|_{q,(-b,-a),\tilde{\nu}}.
\end{align*}
Consequently, inequality \eqref{eq.1.1.4} holds for all $g \in
B^+(I)$ if and only if the inequality
\begin{equation}\label{eq.10.1.1}
\|\tilde{g}\tilde{w}\|_{p,(-b,-a),\tilde{\mu}} \le c\,
\|\tilde{u}(x)\|\tilde{g}\|_{\i,(x,-a),\tilde{\mu}}\|_{q,(-b,-a),\tilde{\nu}}
\end{equation}
holds for all $\tilde{g} \in B^+(\tilde{I})$.

{\rm (i)} Let $0 < q \le p \le +\i$. Since
$\|\tilde{u}\|_{q,(-b,x],\tilde{\nu}} = \|u\|_{q,[-x,b),\nu}$ if $x
\in (-b,-a)$, we deduce from Theorem \ref{thm.1.1.1} that inequality
\eqref{eq.1.1.4} holds on $B^+ (I)$ if and only if
\begin{equation}\label{eq.10.1.2}
\sup_{x \in (-b,-a)} \|\tilde{w}\|_{p,(-b,x],\tilde{\mu}} \,
\|\tilde{u}\|_{q,(-b,x],\tilde{\nu}}^{-1} < \i.
\end{equation}
However, using \eqref{eq.10.2.1} and \eqref{eq.10.2.2}, we see that
condition \eqref{eq.10.1.2} coincides with \eqref{eq.1.1}.

{\rm (ii)} Let $0 < p < q < +\i$ and $1 / r = 1 / p - 1 / q$. By
Theorem \ref{thm.1.1.1}, inequality \eqref{eq.1.1.4} holds for all
$g \in B^+(I)$ if and only if
\begin{align}
A_2 : = & \left( \int_{(-b,-a)}
\|\tilde{w}\|_{p,(-b,x],\tilde{\mu}}^{r} \, d \left( -
\|\tilde{u}\|_{q,(-b,x],\tilde{\nu}}^{-r} \right) \right)^{1 / r}
\notag \\
& + \|\tilde{w}\|_{p,(-b,-a),\tilde{\mu}}
\|\tilde{u}\|_{q,(-b,-a),\tilde{\nu}}^{-1} < \i. \label{eq.10.1.3}
\end{align}
It is clear that
\begin{equation}\label{eq.10.1.4}
\|\tilde{w}\|_{p,(-b,-a),\tilde{\mu}} = \|w\|_{p,(a,b),\mu}
\qq\mbox{and}\qq \|\tilde{u}\|_{q,(-b,-a),\tilde{\nu}} =
\|u\|_{q,(a,b),\nu}.
\end{equation}
Moreover, by the definition of the Lebesgue-Stieltjes integral,
\begin{equation}\label{eq.10.1.5}
\int_{(-b,-a)} \|\tilde{w}\|_{p,(-b,x],\tilde{\mu}}^{r} d \left( -
\|\tilde{u}\|_{q,(-b,x],\tilde{\nu}}^{-r} \right) = \int_{(-b,-a)}
\|\tilde{w}\|_{p,(-b,x],\tilde{\mu}}^{r} d \tilde{\la} = \,: \,D,
\end{equation}
where $\tilde{\la}$ is the non-negative Borel measure associated to
the non-decreasing and right-continuous function $\tilde{\vp}(x) : =
- \|\tilde{u}\|_{q,(-b,x],\tilde{\nu}}^{-r}$, $x \in (-b,-a)$, that
is,
$$
\tilde{\la} ((\tilde{\a},\tilde{\b}]) = \tilde{\vp} (\tilde{\b}) -
\tilde{\vp} (\tilde{\a}) \qq \mbox{for any}\qq
(\tilde{\a},\tilde{\b}] \subset (-b,-a).
$$
Since, by \eqref{eq.10.2.1},
$$
\|\tilde{w}\|_{p,(-b,x],\tilde{\mu}}^{r} = \|w\|_{p,[x,b),\mu}^r
\qq\mbox{for all}\qq t \in (-b,-a),
$$
we obtain from \eqref{eq.10.2.1} that
\begin{equation}\label{eq.10.1.6}
D = \int_{(-b,-a)} \|\tilde{w}\|_{p,(-b,x],\tilde{\mu}}^{r} d
\tilde{\la} = \int_{(a,b)} \|w\|_{p,[x,b),\mu}^r\,d\la,
\end{equation}
where $\la (E) = \tilde{\la}(-E)$ if $E$ is a Borel subset of $I$.
In particular, if $[\a,\b) \subset (a,b)$, then
\begin{align*}
\la ([\a,\b)) & = \tilde{\la}((-\b,-\a]) = \tilde{\vp} (-\a) -
\tilde{\vp} (-\b) \\
& = - \|\tilde{u}\|_{q,(-b,-\a],\tilde{\nu}}^{-r} +
\|\tilde{u}\|_{q,(-b,-\b],\tilde{\nu}}^{-r} \\
& = - \|u\|_{q,[\a,b),\nu}^{-r} + \|u\|_{q,[\b,b),\nu}^{-r}.
\end{align*}
That means that the non-negative Borel measure $\la$ is associated
to the non-decreasing and left-continuous function $\vp$ given on
$(a,b)$ by
$$
\vp (x) : = \|u\|_{q,[x,b),\nu}^{-r}, \qq x \in (a,b).
$$
Consequently,
\begin{equation}\label{eq.10.1.7}
\int_{(a,b)} \|w\|_{p,[x,b),\mu}^r\,d\la = \int_{(a,b)}
\|w\|_{p,[x,b),\mu}^r\,d \left( \|u\|_{q,[x,b),\nu}^{-r}\right).
\end{equation}
The result now follows from \eqref{eq.10.1.3}-\eqref{eq.10.1.7}.

{\rm (iii)} Let $0 < p < +\i$, $q = +\i$. By Theorem
\ref{thm.1.1.1}, inequality \eqref{eq.1.1.4} holds for all $g \in
B^+(I)$ if and only if
\begin{equation}\label{eq.10.1.8}
\int_{(-b,-a)} \left(
\frac{\tilde{w}(x)}{\|\tilde{u}\|_{\i,(-b,x),\tilde{\nu}}}\right)^p\,d\tilde{\mu}(x)
< \i.
\end{equation}
By \eqref{eq.10.2.2} and \eqref{eq.10.2.1}, we have that
$\|\tilde{u}\|_{\infty,(-b,x),\tilde{\nu}} =
\|u\|_{\infty,(-x,b),\nu}$
\begin{align*}
\int_{(-b,-a)} \left(
\frac{\tilde{w}(x)}{\|\tilde{u}\|_{\i,(-b,x),\tilde{\nu}}}\right)^p\,d\tilde{\mu}(x)
& = \int_{(-b,-a)} \left(
\frac{\tilde{w}(x)}{\|u\|_{\infty,(-x,b),\nu}}\right)^p\,d\tilde{\mu}(x)
\\
& =  \int_{(a,b)} \left(
\frac{w(x)}{\|u\|_{\i,(x,b),\nu}}\right)^p\,d\mu(x),
\end{align*}
and we see that \eqref{eq.10.1.8} coincides with \eqref{eq.1.3}.

The equivalency in \eqref{eq.1.3} can be shown by the same
argumentations as in the case (i) and (ii) or by an integration by
parts formula.
\end{proof}

\begin{rem}
Let $q < +\i$. Then
$$
\|u\|_{q,[x-,b),\nu} : = \lim_{s\rw t-} \|u\|_{q,[s,b),\nu} =
\|u\|_{q,[x,b),\nu} \qq\mbox{for all}\qq x \in I,
$$
which allows us to combine the cases (ii) and (iii) of Theorem
\ref{thm.1.2.1}:

{$\rm (ii)^{\prime}$} Let $0 < p < q \le +\i$ and $1 / r = 1 / p - 1
/ q$. Then inequality \eqref{eq.1.1.4} holds for all $g \in B^+ (I)$
if and only if
\begin{equation*}
B_2^{\prime} : = \left( \int_{(a,b)} \|w\|_{p,[x,b),\mu}^{r} d
\left( \|u\|_{q,[x-,b),\nu}^{-r} \right) \right)^{1 / r} +
\|w\|_{p,(a,b),\mu} \|u\|_{q,(a,b),\nu}^{-1} < \i.
\end{equation*}
The best possible constant $c$ in \eqref{eq.1.1.4} satisfies $c \ap
B_2^{\prime}$.
\end{rem}

\section{Inequalities involving three measures}\label{s.5}

Now, we consider the inequality
\begin{equation}\label{eq.9.1.1}
\|g\|_{p,(a,b),\la} \le c \|u(x)
\|g\|_{\i,S_x,\mu}\|_{q,(a,b),\nu},\qq g \in B^+(I),
\end{equation}
for all non-negative Borel measurable functions $g$ on the interval
$(a,b) \subseteq \R$, where $0 < p \le +\i$, $0 < q \le +\i$, $\la$,
$\mu$ and $\nu$ are non-negative Borel measures on $(a,b)$, $u$ is a
weight function on $I$ and either $S_x = (a,x)$ or $S_x = (x,b)$ for
all $x \in I$.

By the same way as it has been done in \cite{EvGogOp}, it is easy to
show that in order to characterize the validity of \eqref{eq.9.1.1}
it is enough to characterize the validity of the inequality
\begin{equation}\label{eq.9.1.2}
\|gw\|_{p,(a,b),\mu} \le c \|u(x)
\|g\|_{\i,S_x,\mu}\|_{q,(a,b),\nu},\qq g \in B^+(I).
\end{equation}
\begin{thm}\label{thm.9.1.1}
Assume that $0 < p,\,q \le +\i$. Let $\la$, $\mu$ and $\nu$ be
non-negative Borel measures on $I = (a,b) \subseteq \R$ and let $u
\in B^+(I)$. Then inequality \eqref{eq.9.1.1} holds for all $g \in
B^+(I)$ if and only if the measure $\la$ is absolutely continuous
with respect to $\mu$ and inequality \eqref{eq.9.1.2} with $w : =
(d\la / d\mu)^{1/p}$ holds for all $g \in B^+(I)$.
\end{thm}
\begin{proof}
Assume that \eqref{eq.9.1.1} holds for all $g \in B^+(I)$. Let $E
\subseteq I$ be such that $\mu(E) = 0$ and put $g = \chi_E$. Then
$\|g\|_{\i,S_x,\mu} = 0$ for all $x \in I$. Therefore, the
right-hand side of \eqref{eq.9.1.1} is zero, which implies that $\la
(E) = 0$, when $0 < p < +\i$, for $\la (E) = \|g\|_{p,(a,b),\la}^p =
0$, and when $p = +\i$, for $\|g\|_{\i,(a,b),\la} =
\|\chi_E\|_{\i,(a,b),\la} = 0$. Hence, $\la$ is absolutely
continuous with respect to $\mu$, and, by the Radon-Nykodym theorem,
there is $v \in B^+(I)$ such that $d\la = v d\mu$. Putting $w =
v^{1/p}$, we have that $d\la = w^p d\mu$. Consequently, for any $g
\in B^+(I)$, we can rewrite the left-hand side of \eqref{eq.9.1.1}
as
$$
\|g\|_{p,(a,b),\la} = \|gw\|_{p,(a,b),\mu},
$$
and our claim follows.
\end{proof}

{\bf Acknowledgment.} We thank the anonymous referee for his
remarks, which have improved the final version of this paper.

\

\

Rza Mustafayev\\
Department of Mathematics, Faculty of Science and Arts, Kirikkale
University, 71450 Yahsihan, Kirikkale, Turkey\\
E-mail: rzamustafayev@gmail.com\\

Tugce {\"U}nver \\
Department of Mathematics, Faculty of Science and Arts, Kirikkale
University, 71450 Yahsihan, Kirikkale, Turkey\\
E-mail: tugceunver@gmail.com \\

\end{document}